\documentclass[12pt]{amsart}
\usepackage{amsmath}
\usepackage{amsfonts}
\usepackage{amssymb}
\usepackage{graphicx} 
 
 \newtheorem{thm}{Theorem}[section]
 \newtheorem{lem}[thm]{Lemma}
 \newtheorem{exam}[thm]{Example}
 \newtheorem{prop}[thm]{Proposition}
 \def\Hol{\mathop{\rm Hol}\nolimits}

\newcommand{\C}{{\mathbb C}}   
 \newcommand{\R}{{\mathbb R}}           
\newcommand{\N}{{\mathbb N}}        
\newcommand{\D}{{\mathbb D}}

\author{W. Arendt}
\address{Wolfgang Arendt, Institute of Applied Analysis, University of Ulm. Helmholtzstr. 18, D-89069 Ulm (Germany)} 
\email{wolfgang.arendt@uni-ulm.de}

\author{I. Chalendar}
\address{Isabelle Chalendar, Universit\'e Paris Est Marne-la-Vall\'ee, 5 bld Descartes, Champs-sur-Marne, 77454 Marne-la-Vall\'ee, Cedex 2  (France)}
\email{isabelle.chalendar@u-pem.fr}
\title[Generators inducing holomorphic semiflows]{Generators of semigroups on Banach spaces inducing holomorphic semiflows}
\begin{document}	

\begin{abstract}
Let $A$  be the generator of a $C_0$-semigroup  $T$ on a   Banach space of analytic functions on the open unit disc.
If $T$ consists of composition operators, then there exists a holomorphic function  $G:\D\to\C$	such that $Af=Gf'$ with maximal domain. The aim of the paper is the study of the  reciprocal implication. %
\end{abstract}
\maketitle
\section{Introduction} 
Let $\Omega$ be a non-empty open and connected subset of $\C$ and let  $X$ be a complex Banach space which embeds continuously in the Fr\'echet space of all holomorphic functions on $\Omega$ which we denote by $\Hol(\Omega)$. Given a holomorphic global semiflow $\varphi:\R^+\times \Omega \to \Omega$, in many cases, $T(t)f=f\circ \varphi(t,\cdot)$ defines a 
$C_0$-semigroup $T$  on $X$. Its infinitesimal generator
 $A$ is of the form
 \begin{equation}\label{eq:first}
  D(A)=\{f\in X:Gf'\in X\}\mbox{ and }Af=Gf',
  \end{equation}
  where $G\in \Hol(\Omega)$ is the infinitesimal generator of the semiflow.    A non-exhaustive list of  examples of  Banach spaces on which global holomorphic semiflows induce a $C_0$-semigroup of composition operators can be found in \cite{BP,CM,team,GY}. Classical examples treated in the literature for $\Omega=\D$ are weighted Hardy, Dirichlet and Bergman spaces, the disc algebra, VMOA and the little Bloch space. \\

  In this paper, we study the reciprocal  implication, namely: \\
  \begin{center}
  given $G\in \Hol (\Omega)$, if $T$ is a $C_0$-semigroup on $X$ whose generator 
  $A$ satisfies $D(A)=\{f\in X:Gf'\in X\}\mbox{ and }Af=Gf'$, is $T$ necessarily a semigroup of composition operators? 
  \end{center}
\vspace{0.2cm}

In Section~\ref{sec:3} we provide an example to show that without an extra hypothesis on $X$ the answer is no in general.  When $X$ is the hilbertian Hardy space or Dirichlet space on $\D$, provided $T$ is  quasicontractive, the answer is positive \cite{ACP2015,ACP2016}. The quasicontractivity condition can be removed. If  $\Hol(\overline{\D})$ embeds continuously in $X$, E. Gallardo-Guti\'errez and D. Yakubovitch \cite{GY} gave a positive answer to this problem.\\

In the present paper we give a positive answer to this problem with different conditions on the space $X$ and we 
use an approach different from \cite{GY}. In fact, we are inspired by the theory of Ordinary Differential Equations, and this is one reason to consider general domains in $\mathbb C$. Of course, by 
the Riemann mapping theorem, one can reduce the problem to the disc if $\Omega$ is simply connected and different from $\mathbb C$ (see more comments in Section~\ref{sec:final}).  

The main result of Section~\ref{sec:3} is Theorem~\ref{th:2} where we give a positive answer to our problem provided $X$ satisfies the "evaluation condition" $(E)$. 

In Section~\ref{sec:4bis} we restrict ourselves to bounded domains  $\Omega$ whose boundary satisfies a weak regularity condition. A positive answer is given assuming that   
$X$ contains the identity, that the domain of holomorphy $\Omega$ is \textit{maximal} for $X$ as well as  a density condition $(D)$ which is automatic when the polynomials are dense in $X$.   

Our approach consists in considering the Cauchy problem 
\[u'(t)=G(u(t)),\mbox{  }u(0)=z\in \Omega,\]
where $G\in \Hol(\Omega)$ is an arbitrary function.   
The main point is to show that its solutions are \textit{global} (i.e. for each $z\in \Omega$ the solution exists on $[0,\infty)$), whenever the operator $A$ given by (\ref{eq:first}) generates a $C_0$-semigroup on $X$.

It turns out that for bounded domains the results established in Section~\ref{sec:4bis}  are most efficient. Indeed, for the classical Banach spaces such as Hardy, Dirichlet, Bergman, VMOA and little Bloch, we obtain a complete picture which is presented in the concluding Section~\ref{sec:5bis}.  

\section{Holomorphic semiflows}\label{sec:2}
In this section we put together the essential properties of holomorphic semiflows. 
Let $\Omega\subset\C$ be a non-empty open subset of $\C$
and $G\in\Hol(\Omega)$. 
We consider the following initial value problem:
\begin{equation}\label{eq:pb}
u'(t)=G(u(t))\mbox{ with }u(0)=z,
\end{equation} 
for $z\in \Omega$. By \[\varphi(\cdot,z):[0,\tau(z))\to \Omega \] we denote the unique maximal solution of (\ref{eq:pb}). Here $\tau(z)>0$ and if $\tau(z)<\infty$ then  there exists an increasing sequence $(t_n)_n$ in 
$[0,\tau(z))$ such that  $\lim_{n\to \infty}t_n=\tau(z)$ and $(\varphi(t_n,z))_n$ converges either to a point on the boundary of $\Omega$ or there exists such a sequence $(t_n)_n$ such that $\lim_{n\to \infty} |\varphi(t_n,z)|=\infty$. Of course the last case can only occur if $\Omega$ is unbounded. These are standard results in Ordinary Differential Equations, we refer to \cite[Chap.8, \S 5-8]{ms74}.
We call $\phi$ \textit{the semiflow generated by} $G$. It satisfies the semigroup property   
\begin{equation}\label{eq:sg}
\varphi(t+s,z)=\varphi(t,\varphi(s,z)),
\end{equation}
where $t+s<\tau(z)$. Note that this last condition holds if and only if $s<\tau(z)$ and $t<\tau(\varphi(s,z))$. We call the semiflow \textit{global} if $\tau(z)=\infty$ for all $z\in\Omega$. In  that case (\ref{eq:sg}) holds for all $t,s\geq 0$ and $z\in\Omega$.   
Moreover, $\varphi:[0,\infty)\times\Omega\to \Omega$ is continuous (see \cite[Chap.8, \S 7, Thm. 2]{ms74}) and $\varphi(t,\cdot)\in \Hol(\Omega)$ for all $t>0$ (see the proof of \cite[Chap.~1, Thm.~9]{MW}). See also \cite[Sec. 2]{AERS} for related results.

Even though phrased in a slightly different way, the main result of Berkson and Porta is the following remarkable characterization.

\begin{thm}[\cite{BP}, Theorem 3.3]\label{th:bp}
	Let $\Omega=\D$. The semiflow $\varphi$ is global if and only if there exist $b\in\overline{\D}$ and a holomorphic function $F:\D\to \C$ such that $\mathop{Re}(F(z))\geq 0$ for all $z\in\D$ and 
	\[G(z)=(b-z)(1-\overline{b}z)F(z),\]
	for all $z\in\D$.   
\end{thm}

\section{The Evaluation Condition}\label{sec:3}
 Let $\Omega\subset\C$ be a connected open set
 and let $X$ be a subspace of ${\Hol}(\Omega)$ which is a Banach space endowed with its proper norm $\|\cdot \|_X$.   The following proposition describes 
the compatibility of the norm  $\|\cdot \|_X$ with the topology of  ${\Hol}(\Omega)$. 

If $z\in \Omega$, then we denote by $\delta_z:{\Hol}(\Omega)\to \C$ the evaluation at $z$; i.e. 
\[\langle \delta_z, f\rangle:=f(z),\quad f\in {\Hol}(\Omega).\]
We write $\delta_z\in X'$ if there exists $c\geq 0$ such that 
\[| \langle \delta_z,f\rangle |\leq c\|f\|_X\quad\text{for all } f\in X.\]

\begin{prop}[Prop. 2.1 in \cite{delhi1}]\label{prop2.1}
	The following assertions are equivalent:
	\begin{itemize}
		\item[(i)] $\delta_z\in X'$ for all $z\in\Omega$;
		\item[(ii)] $f_n\to f$ in $X$ implies that $f_n^{(m)}(z)\to f^{(m)}(z)$ uniformly on each compact subset of $\Omega$ and all $m\in \N_0$, where $\N_0=\N\cup\{0\}$;
		\item[(iii)] there exists a compact set $K\subset \Omega$ such that $\{z\in K:\delta_z\in X'\}$ is infinite. 
	\end{itemize}	
\end{prop}
If these equivalent conditions are satisfied, we write $X \hookrightarrow \Hol(\Omega)$ and call $X$ a \textit{Banach space of holomorphic  functions}. 

In this section we assume furthermore the following \textit{evaluation condition}:
\[{(E)}
\hspace{0.5cm}\mbox{If } z_n\in \Omega \mbox{ such that }z_n\to z\in \overline{\Omega}\cup\{\infty\} \mbox{ and } \lim_{n\to\infty}f(z_n)\mbox{ exists in }\C \]
$
\hspace{0.5cm}\mbox{ for all }  f\in X,\mbox{ then }z\in \Omega.
$\\

 Here we use the one-point compactification $\C\cup\{\infty\}$ of $\C$. Of course if $\Omega$ is bounded this is not needed  since there is no sequence $(z_n)_n\subset\Omega$ converging to $\infty$. \\
   
 Classical examples of such spaces are  provided in Example~\ref{ex1}.\\
 
 Let $G\in \Hol(\Omega)$ and denote by $\varphi$ the semiflow generated by $G$. Let $X\hookrightarrow \Hol(\Omega)$ be a Banach space satisfying the evaluation condition $(E)$, and let $T$ be a $C_0$-semigroup on $X$ with generator $A$.
  

\begin{thm}\label{th:2}
Assume that $Af=Gf'$ for all $f\in D(A)$. Then the semiflow $\varphi$ is global and 
\[T(t)f)(z)=f(\varphi(t,z))\]
for all $f\in X,t\geq 0,z\in\Omega$. \\
Moreover the domain of $A$ is maximal, i.e., $A$ is given by 
\[D(A)=\{f\in X:Gf'\in X\},\hspace{0.5cm}Af=Gf'.\]	
\end{thm}
For the proof of Theorem~\ref{th:2}, we use  the following lemmas which will also be fundamental for Section~\ref{sec:4bis}. 

In the following lemma
 $A'$ denotes the adjoint of the generator $A$.

\begin{lem}\label{lem:1}
	 For all $z\in \Omega$, $\delta_z\in D(A')$ and $A'\delta_z=G(z)\delta'_z$, where 
	\[\langle \delta_z',f\rangle =f'(z).\]   
\end{lem}
\begin{proof}
Let $f\in D(A)$. Then 
\[\langle Af,\delta_z\rangle =G(z)f'(z)=\langle f, G(z)\delta_z'\rangle.\]	
\end{proof}	
\begin{lem}\label{lem:2}
Let $f\in X$ and set $u(t,z)=(T(t)f)(z)$. Then \[u_t(t,z)=G(z)u_z(t,z)\]
 for all $t\geq 0,z\in \Omega$ and where, as usual, $u_z:=\frac{d}{dz}u$,  $u_t:=\frac{d}{dt}u$. 
\end{lem}
\begin{proof}
	Note that, using Lemma~\ref{lem:1}, we have:\\
\[	
\begin{array}{ll}  u_t(t,z)  &= \frac{d}{dt}\langle T(t)f,\delta_z\rangle = 
	  \langle T(t)f,A'\delta_z\rangle=G(z)(T(t)f)'(z)\\
	  & =G(z)u_z(t,z). 
\end{array}
\]	
\end{proof}

\begin{lem}\label{lem:3}
	Let $f\in X$. Then 
	\[(T(t)f)(z)=f(\varphi(t,z))\]
for all $z\in\D$ and $0\leq t<\tau(z)$.
\end{lem}
\begin{proof}
	Let  $u(t,z)=(T(t)f)(z)$ for a fixed $z\in\Omega$ and $0\leq t<\tau(z)$. Now define $w$ on $[0,t]$ by 
	\[w(s)=u(t-s,\varphi(s,z)).\]
	Then 
	\[
	\begin{array}{ll}
	w'(s) &= -u_t(t-s,\varphi(s,z))+u_z(t-s,\varphi (s,z))\varphi_t(s,z)\\
	  & = -u_t(t-s,\varphi(s,z))+u_z(t-s,\varphi (s,z))G(\varphi(s,z))\\
	   & = 0
	  \end{array}
	\]
	by Lemma~\ref{lem:2}. It follows that $w(t)=w(0)$ with $w(0)=u(t,z)$ and 
	\[w(t)=u(0,\varphi(t,z))=(T(0)f)(\varphi(t,z))=f(\varphi(t,z)).\]
	 \end{proof}
 So far we did not use hypothesis $(E)$, but it is needed in the following lemma. 
 
\begin{lem}\label{lem:4}
 One has $\tau(z)=\infty$ for all $z\in \Omega$, that is, the semiflow is global on $\Omega$.  
\end{lem}
\begin{proof}
	Assume that there exists $z\in\Omega$ such that $\tau(z)<\infty$. Then there exist an increasing  sequence $(t_n)_n$ converging to $\tau(z)$ and $\alpha\in\partial\Omega\cup\{\infty\}$ such that $\lim_{n\to\infty}\varphi(t_n,z)=\alpha$.
	Let $f\in X$. Then  
	\[\lim_{n\to\infty}f(\varphi(t_n,z))=\lim_{n\to \infty} ((T(t_n)f)(z)=(T(\tau(z))f(z), \]
which contradicts $(E)$. 
\end{proof}	

\begin{proof}[Proof of Theorem~\ref{th:2}]
It follows from Lemma~\ref{lem:3} and Lemma~\ref{lem:4}	that $T(t)$ is given by 
\[T(t)f=f\circ\varphi(t,\cdot) \mbox{ for all }t\geq 0.\]
Thus, it remains to show that the domain of $A$ is maximal. Let $f\in X$ such that $g=Gf'\in X$. Then, for $z\in \Omega$,
\[\frac{d}{ds}f(\varphi(s,z))=g(\varphi(s,z))\mbox{ for all }s\geq 0.\]
Hence, for $t>0$, 
\[\int_0^t(T(s)g)(z)ds = \int_0^t g(\varphi(s,z))ds=f(\varphi(t,z))-f(z)=(T(t)f)(z)-f(z).\]
Therefore we have
\[\frac{1}{t}\int_0^tT(s)gds=\frac{1}{t}(T(t)f-f).\]
Passing to the limit as $t\downarrow 0$ shows that $f\in D(A)$ and $Af=g$. 
\end{proof}	
It is remarkable that the domain is always maximal and that for the complex differential operator boundary conditions never occur, in contrast to  the real analogue.
\begin{exam}
	On $X=L^1(0,1)$, let the operator $A_{max}$ be given by $D(A_{max})=AC[0,1]$ (the space of all absolutely continuous functions on $[0,1]$) and $A_{max} f=f'$. Then $A_{max}$ is not a generator. Indeed, for example, the restriction $A$ of $A_{max}$ given by 
	\[D(A):=\{f\in AC[0,1]:f(1)=0\}\mbox{ and } Af=f'\]
	generates the $C_0$ semigroup $T$ given by $T(t)f=f\circ \varphi_t$ where 
	\[\varphi(t,x)=\left\{    \begin{array}{lcl}
	x+t & \mbox{if} &x+t\leq 1\\
	0    & \mbox{if} & x+t>1.
	\end{array}\right.\] 
\end{exam}

Finally, note that it is not possible to remove the hypothesis $(E)$, as clarified by the following result. 
\begin{prop}\label{prop:ce}
There exist a Banach space $X\hookrightarrow \Hol(\D)$ and $G\in \Hol(\D)$ such that the operator $A$ on $X$ defined by 
\[D(A):=\{f\in X:Gf'\in X\}\mbox{ and }Af=Gf'\]
generates a $C_0$-semigroup $T$, but $T$ does not consist of composition operators.  
 \end{prop}
\begin{proof}
Let $\Omega:=\{z\in\C:|z|<2\}$ and let \[X:=\{g\in \Hol(\D):\exists f\in H^2(\Omega), f_{|\D}=g\}.\]
By the uniqueness theorem for holomorphic functions, the mapping $Uf:=f_{|\D}$ is a bijection from $H^2(\Omega)$ to $X$, and then 
\[\|Uf\|_X:=\|f\|_{H^2(\Omega)}\]
defines a complete norm on $X$, with $X \hookrightarrow \Hol(\D)$. Moreover, the choice of the norm makes $U$ an isometric isomorphism from $H^2(\Omega)$ to $X$.  
 
 Let $F:\Omega\to \C$ be a holomorphic function such that $\mathop{Re}( F(z))\geq 0$ for all $z\in \Omega$. Choose $b\in \Omega$ such that $|b|>1$ and define $\widetilde{G}\in \Hol(\Omega)$ by 
 \[\widetilde{G}(z)=F(z)(\overline{b}z-1)(z-b)\mbox{ for all }z\in\Omega.\]
 Applying Theorem~\ref{th:bp} to $\Omega$ instead of $\D$, it follows that $\widetilde{G}$ generates a global semiflow $\varphi$ on $\Omega$, i.e. 
 $\varphi:[0,\infty)\times \Omega\to\Omega$ satisfies
 \[\varphi_t(t,z)=G(\varphi(t,z))\mbox{ and }\varphi(0,z)=z.\]
 Moreover $\varphi(t,\cdot)$ is holomorphic on $\Omega$ for all $t\geq 0$, and the Denjoy-Wolff point of $\varphi(t,\cdot)$ is $b$, i.e.
 \[\lim_{t\to\infty}\varphi(t,z)=b\mbox{ for all }z\in \Omega. \] 
 In particular, the semiflow does not leave invariant the open unit disc. By \cite[Theorem 3.4]{ACP2015},  
 \[\widetilde{T}(t)f=f\circ \varphi(t,\cdot)\]
 defines a $C_0$-semigroup on $H^2(\Omega)$ whose generator is given by 
 \[D(\widetilde{A})=\{f\in H^2(\Omega):Gf'\in H^2(\Omega)\}\mbox { and }\widetilde{A}f= \widetilde{G}f'.\]
 Since $U$ is an isometric isomorphism, 
 \[T(t):=U\widetilde{T}(t)U^{-1}\]
 defines a $C_0$-semigroup $T$ on $X$, whose generator is given by 
 \[D(A)=\{g\in X:U^{-1}g\in D(\widetilde{A})\mbox{ and }Ag=U\widetilde{A}U^{-1}g.\] 
 Let $G=\widetilde{G}_{|\D}$. \\
 
 \textbf{Claim}: $D(A)=\{g\in X:Gg'\in X\}$ and $Ag=Gg'$. 
 
 \textit{Proof of the Claim.} 
 Indeed, let $g\in D(A)$. Then $U^{-1}g\in D(\widetilde{A})$, i.e. there exists $f\in H^2(\Omega)$ such that 
 $\widetilde{G}f'\in H^2(\Omega)$ and $f_{|\D}=g$. Moreover, 
 \[Ag=U\widetilde{A}U^{-1}g=(\widetilde{G}f')_{|\D}=Gg'.\]
 Conversely, let $g\in X$ such that $Gg'\in X$. Then there exists $f\in H^2(\Omega)$  such that $f_{|\D}=g$ and there exists $h\in H^2(\Omega)$ such that $h_{|\D}=Gg'$. Thus, $\widetilde{G}f'\in \Hol(\Omega)$ and 
 \[\widetilde{G}f'_{|\D}=Gg'=h_{|\D}.\]
 It follows that $\widetilde{G}f'=h$, and then $\widetilde{G}f'\in H^2(\D)$. Consequently, $f\in D(\widetilde{A})$ and so $g=Uf\in D(A)$ and 
 \[Ag=U\widetilde{A}U^{-1}g=U\widetilde{A}f=\widetilde{G}f'_{|\D}=Gg'.\]
The semigroup $T$ is given for $g\in X$ by 
\[(T(t)g)(z)=f(\varphi(t,z)) \mbox{ for all }z\in\D ,\]
where $f\in H^2(\Omega)$ such that $f_{|\D}=g$. But $T(t)$ is not a composition operator since the semiflow 
does not leave $\D$ invariant. 
\end{proof}
\begin{exam}\label{ex1} 
\noindent	
\begin{itemize}
	\item[a)] The spaces $H^p(\Omega)$, $1\leq p<\infty$ satisfy Condition $(E)$ for $\Omega=\D$ or $\Omega=\C^+$ (the right-half plane).
	\item[b)] The disc algebra does not satisfy Condition $(E)$.  
	\item[c)] Let $\beta=(\beta_n)_{n\geq 0}$ be a sequence of positive reals such that\\ $\liminf_{n\to\infty}\beta_n^{1/n}\geq 1$,  $1\leq p<\infty$ and let 
	\[H^p(\beta):=\left\{f\in\Hol(\D):\left( \frac{\widehat{f}(n)}{n!} \right)_{n\geq 0}\in \ell^p(\beta)\right\}, \]
	where 
	\[\ell^p(\beta):=\{ (a_n)_{n\geq 0}:\sum_{n\geq 0}|a_n|^p\beta^p_n<\infty\}. \]
	Then $H^p(\beta)$ is a Banach space and $H^p(\beta)\hookrightarrow\Hol(\D) $ (see \cite[Prop. 2.1]{delhi1}). By \cite[Prop.~2.5 and Remark~2.6]{delhi1}, the following assertions hold:\\
	1) $p>1$ and let $q$ such that $1/p +1/q=1$. Then $H^p(\beta)$ satisfies Condition $(E)$ if and only if $\sum_{n\geq 0}\beta_n^{-q}=\infty$. \\
	2) Let $p=1$. Then $H^1(\beta)$ satisfies Condition $(E)$ if and only if $\inf_{n\geq 0}\beta_n=0$.

\end{itemize}	
	
\end{exam}

\section{Banach spaces with maximal domain}\label{sec:4bis}  
In this section we consider other conditions which are satisfied for the disc algebra (in contrast to $(E)$). 
The main assumption (maximality of the domain, see below) is very natural in view of the construction given in Proposition~\ref{prop:ce}    .\\

Let $\Omega \subset \C$ be an open connected bounded non-empty set. We assume that the boundary satisfies a weak regularity condition, namely  that if $D(z,r)\setminus \{z\}\subset \Omega$, then $z\in \Omega$ for each open disc $D(z,r)$ centered at $z$ and of radius $r>0$.  This is certainly the case if $\Omega$ is simply connected. It is not difficult to see that this topological condition  implies that the boundary of $\Omega$ has no isolated point. 

Denote by $e_1$ the identity function, i.e.   $e_1(z)=z$ for all $z\in\Omega$. \\
We say that the domain $\Omega$ is \textit{maximal} for $X$ if for each $w\in\partial\Omega$ and each $\varepsilon>0$, there exists 
$f\in X$ which does not have a holomorphic extension to $\Omega\cup D(w,\varepsilon)$.  Here we denote 
by $D(w,\varepsilon):=\{z\in\C :|w-z|<\varepsilon\}$  the disc centered at $w$ and of radius $\varepsilon$. 

We also consider the following \textit{density condition}:
 
\[\begin{array}{ll}
(D) & \mbox{For all }w\in\partial \Omega, \mbox{ there exists }\varepsilon>0\mbox{ such that  the space}\\
 &  \{ f\in X: \exists \widetilde{f}\in  \Hol(\Omega\cup D(w,\varepsilon)) , \widetilde{f}_{|\Omega}=f \} \mbox{ is dense in }X. 
 \end{array}
 \]

\vspace{0.2cm}
Now, let $G\in \Hol(\Omega)$ and assume that $T$ is a $C_0$-semigroup whose generator $A$ is given by 
\[Af=Gf'\mbox{ for all }f\in D(A).\]
\begin{thm}\label{th:main}
	Let $X$ be a Banach space such that $X\hookrightarrow \Hol(\Omega)$ where $\Omega$ is maximal for $X$.  Suppose that $e_1\in X$  and that $X$ satisfies the condition $(D)$. Then  
	the semiflow $\varphi$ associated with $G$ is global and 
	\[(T(t)f)(z)=f(\varphi(t,z))\]
	for all $f\in X,z\in\Omega$ and $t\geq 0$. Moreover, 
	\[D(A)=\{f\in X:Gf'\in X\}.\]
\end{thm}  
\begin{proof}
We consider the semiflow also for negative time, i.e. for $z\in \Omega$, 
\[\varphi(\cdot,z):(\tau^-(z), \tau^+(z))\to \Omega,\]
where 	$-\infty \leq \tau^-(z)<0<\tau^+(z)\leq \infty,$
\[  \varphi_t(t,z)=G(\varphi(t,z))\mbox{ for }t\in (\tau^-(z),\tau^+(z))\mbox{ and }\varphi(0,z)=z.\]
We know from Section~\ref{sec:3} that 
\begin{equation}\label{eq:new}
(T(t)f)(z)=f(\varphi(t,z))\mbox{ for all }z\in\Omega,0\leq t<\tau^+(z), f\in X.
\end{equation}
Let $u(t,z):=(T(t)e_1)(z)$ for all $t\geq 0$ and $z\in\Omega$. Then $u(t,\cdot)\in\Hol(\Omega)$ since it is in $X$,  and moreover, by Lemma~\ref{lem:3}, 
\[u(t,z)=\varphi(t,z)\mbox{ if }0\leq t<\tau^+(z).\]	
	
Since $T(t)e_1\to e_1$ in $X$, and therefore in $\Hol(\Omega)$ as $t\downarrow 0$, 	there exists $\widetilde{t_0}>0$ such that $T(t)e_1$ is not constant whenever $0\leq t\leq \widetilde{t_0}$. 
Assume that there exists $z_0\in\Omega$ such that $t_0:=\tau^+(z_0)<\widetilde{t_0}$,
 It follows that $w_0=u(t_0,z_0)\in\partial\Omega$. \\
 Using $(D)$, choose $\varepsilon>0$ such that 
 \[X_0:=\{f\in X: \exists \widetilde{f}\in  \Hol(\Omega\cup D(w_0,\varepsilon)) , \widetilde{f}_{|\Omega}=f \} \mbox{ is dense in }X. \]
 Let ${\mathcal U}$ be an open connected  neighborhood of $z_0$ such that $u(t_0,z)\in D(w_0,\varepsilon)$ for all $z\in {\mathcal U}$. Now set $\Psi(z)=u(t_0,z)$ for all $z\in{\mathcal U}$ and ${\mathcal V}:=\Psi({\mathcal U})$. 
 Since $\Psi$ is not constant, ${\mathcal V}$ is an open connected subset of $\C$ containing $w_0$. 
 
 Let $f\in X_0$ and define the functions $g$ and $h$ on $\mathcal U$ by 
 \[g(z)=\widetilde{f}(u(t_0,z))\mbox{ and }h(z)=(T(t_0)f)(z). \]
 The set ${\mathcal U}_0:=\{z\in {\mathcal U}:\tau^+(z)>t_0\}$ is open by \cite[Chap.8, \S 7, Thm. 2]{ms74}. Moreover  ${\mathcal U}_0$ is  non-empty since $\varphi(-\delta,z_0)\in{\mathcal U}_0$ for $\delta>0$ sufficiently small. Note that for all $z\in {\mathcal U}_0$, $u(t_0,z)=\varphi(t_0,z)$ and thus, by (\ref{eq:new}), $g$ coincides with $h$ on ${\mathcal U}_0$. By the uniqueness theorem for holomorphic functions $g=h$ on $\mathcal U$. In particular, if $z\in {\mathcal U}$ is such that
  $u(t_0,z)\in\Omega$, then 
 \[(T(t_0)f)(z)=f(u(t_0,z))    \mbox{ for all }f\in X_0.\]
 The density of $X_0$ in $X$ implies that, for all $f\in X$ and all $z\in  {\mathcal U}$ such that $u(t_0,z)\in {\Omega}$,
 \begin{equation}\label{eq:density}
 (T(t_0)f)(z)=f(u(t_0,z)). 
 \end{equation} 
 We may assume that $\overline{{\mathcal U}}\subset \Omega$. 
Since $\Psi=u(t_0,\cdot)$ is not constant, the set $\{z\in{\mathcal U}:\Psi'(z)=0\}$ is finite. Since, by our assumption, the boundary does not contain isolated points, $\partial\Omega\cap {\mathcal V}$ is an infinite set. Consequently, there exists $z_1\in{\mathcal U}$ such that $\Psi'(z_1)\neq 0$   and $\Psi(z_1)\in\partial\Omega$. Therefore, there exists an open neighborhood ${\mathcal U}_1$ of $z_1$ such 
\[{\mathcal U}_1\subset {\mathcal U}\mbox{ and }\Psi:{\mathcal U}_1\to \Psi({\mathcal U}_1)=:{\mathcal V}_1 \mbox{ is biholomorphic}.  \]
Note that ${\mathcal V}_1$ is an open neighborhood of $w_1:=\Psi(z_1)\in\partial\Omega$. 

Now, for $f\in X$, define 
\[\widetilde{f}(w)=(T(t_0)f)(\Psi^{-1}(w))\mbox{ for all }w\in {\mathcal V}_1,\]
so that $\widetilde{f}$ is holomorphic on ${\mathcal V}_1$.\\
 We now show that $\widetilde{f}=f$ on ${\mathcal V}_1\cap \Omega$.   	
Let $w\in {\mathcal V}_1\cap {\Omega}$, $z=\Psi^{-1}(w)$. Then $u(t_0,z)=w$ and 
\[\widetilde{f}(w)  =(T(t_0)f)(z)=f(w)\mbox{ by }(\ref{eq:density}).\]
Therefore we have shown that each $f\in X$ has a holomorphic extension to ${\mathcal V}_1$ which contradicts the hypothesis of maximality of $\Omega$ for $X$. \\
 Consequently $\tau^+(z)\geq \widetilde{t_0}$ for all $z\in \Omega$ and $\varphi(t,\cdot)=u(t,\cdot)$ leaves invariant $\Omega$ for all $t\in[0,\widetilde{t_0}]$.\\ This implies that $\tau^+(z)=\infty$ for all $z\in \Omega$. In fact, otherwise, 
 \[t_2:=\inf\{\tau^+(z):z\in\Omega \} \in [\widetilde{t_0},\infty).\]
 Let $t_2-\widetilde{t_0}<t_3<t_2$, $z\in\Omega$, and define 
 \[w(t)=  \left\{    
 \begin{array}{lcl}
 \varphi(t,z) & \mbox{for} & 0\leq t\leq t_3\\
 \varphi(t-t_3,\varphi(t_3,z)) & \mbox{for} & t_3<t<t_3+\widetilde{t_0}.
 \end{array}\right.  \]
 Then $w$ solves (\ref{eq:pb}) on $[0,t_3+\widetilde{t_0}]$. Thus $\tau^+(z)\geq t_3+\widetilde{t_0}>t_2$ for all $z\in\Omega$, contradicting the definition of $t_2$. 
 
 We have shown that the semiflow $\varphi$ is global. Now we conclude as in the proof of Theorem~\ref{th:2}. 
 
 \end{proof}	 

\begin{exam}
	The disc algebra  satisfies the density condition $(D)$ and has $\D$ as maximal domain. For this space also sufficient conditions are known when global semiflows lead to a $C_0$-semigroup (see \cite{team} and the references given there).  
\end{exam}
 Concluding, Theorem~\ref{th:main} allows us to give a complete characterization for all the spaces
  mentioned in the introduction.
\section{Conclusion}\label{sec:5bis}
Let $X$ be one of the following spaces:
\begin{enumerate}
	\item the Hardy spaces $H^p(\D)$, $1\leq p<\infty$,
	\item the Bergman spaces $A^p(\D)$,  $1\leq p<\infty$, 
	\item the Dirichlet space $\mathcal D$,
	\item the space $VMOA$ of all analytic functions on $\D$ of vanishing mean oscillation,
	\item the little Bloch space ${\mathcal B}_0$.
	\end{enumerate}
Then we obtain the following final result. 
\begin{thm} 
	Let $G\in \Hol(\D)$. Consider the operator $A$ on $X$ given by 
	\[D(A):=\{f\in X:Gf'\in X\}\mbox{ and }Af=Gf'.\]
Then the following conditions are equivalent:
\begin{itemize}
	\item[(i)]  $A$ generates a $C_0$-semigroup $T$ on $X$;
	\item[(ii)] the semiflow $\varphi$ generated by $G$ is global.
\end{itemize}
In that case $T$ is given by
 \[(T(t)f)(z)=f(\varphi(t,z)),\]
 for all $t>0$, $z\in\D$, $f\in X$. 
\end{thm}
\begin{proof}
$(ii)\Rightarrow (i)$: we refer to \cite[Section 4]{team}.	\\
$(i)\Rightarrow (ii)$: The space $X$ contains the polynomials as dense subspace in each case. Thus the density condition $(D)$ is satisfied. Moreover, the disc algebra is contained in $X$, when $X\neq {\mathcal D}$, which shows that $\D$ is maximal for $X\neq {\mathcal D}$. For $X={\mathcal D}$, since $f\in {\mathcal D}$ if and only if $f'$ is in the Bergman space $A^2$, $\D$ is also maximal for $\mathcal D$. Now Theorem~\ref{th:main} shows that $(ii)$ holds and that $T$ has the desired from.   
	
\end{proof}	 

\section{Conformal equivalence}\label{sec:final}
The study of (global) semiflows is very advanced on the open unit disc, but also other domains occur (see the monographs \cite{shoik,RS2005}; also Berkson and Porta \cite{BP} consider the disc and the half-plane). 

If the domain is simply connected, for the question we treat here, it can always be reduced to an equivalent problem on the disc. The price to be paid is to lose a natural norm on the given Banach space. We make this more precise.

 Let $\Omega$ be a simply connected domain different from $\mathbb  C$ and let $X(\Omega)\hookrightarrow Hol(\Omega)$ be a Banach space. Consider a conformal mapping $h$ from $\D$ onto $\Omega$. Then $X(\D):=\{ f\circ h:f\in X(\Omega)    \}$ is a Banach space with the transferred norm and $X(\D)\hookrightarrow Hol(\D)$. 
 Then $C_hf=f\circ h$ defines an isomorphism from $X(\Omega)$ onto $X(\D)$. Now let $G\in Hol(\Omega)$ and let $T$ be a $C_0$-semigroup on $X(\Omega)$ whose generator $A$ is given by
 \[  D(A)=\{   f\in X(\Omega)  :Gf'\in X(\Omega), \quad Af=Gf'  \}. \]
 Then $S(t):=C_hT(t)C_h^{-1}$ defines a $C_0$-semigroup on $X(\D)$. Its generator $B$ is given by 
 \[  D(B)=\{  g\in X(\D):C_h^{-1}g\in D(A)   \}, \quad Bg=C_hAC_h^{-1}g.   \]
 One easily computes that 
 \[  D(B)=\{   g\in X(\D)  :Hg'\in X(\D), \quad Bg=Hg'  \}, \]
 where $H=\frac{1}{h'}G\circ h\in Hol(\D)$. 
 
 Thus $B$ has the same form as $A$. Moreover, if $S$ is given by a global semiflow $\varphi$ on $\D$, i.e. $S(t)g=g\circ \varphi(t,\cdot)$, then $T$ is given by the global semiflow $\psi$ on $\D$ defined by 
 \[\psi(t,z)=h(\varphi(t,h^{-1}(z))) \quad  (t>0,z\in\Omega),\quad T(t)f=f\circ \psi(t,\cdot). \]
 Therefore a criterion on $X(\D)$ leads to a criterion on $X(\Omega)$. However, $X(\D)$ might carry a norm which is not easy to compute. 
 
 We give now a more concrete example. Consider $\Omega=\mathbb P$, the upper half-plane. Then $h(z)=i\frac{1+z}{1-z}$ is a conformal mapping from $\D$ to $\mathbb P$. Consider $X(\Omega)=H^p({\mathbb P})$ where $1\leq p<\infty$. Then the transferred norm on $X(\D)$ is given by $\|g\|_{X(\D)}=\|wg\|_{H^p(\D)}$, where $w:\D\to (0,\infty)$ is a weight (see \cite[Sec.2]{cp03}). If we want to use the usual norm on $H^p(\D)$,  we can consider the isomorphism $U:X(\D)\to H^p(\D)$ given by $Ug=wg$.   
     Then $\widetilde{S}(t)=US(t)U^{-1}$ defines a $C_0$-semigroup on $H^p(\D)$. Its generator $\widetilde{B}$ is given by 
     \[ D(\widetilde{B})=\{  f\in H^p(\D):w^{-1}f\in X(\D) \},\quad \widetilde{B}f=Hf'-\frac{w'}{w}Hf.   \]
     It is no longer of the form we consider here and is rather an additive perturbation of our usual form. If $\widetilde{S}$ consists of composition operators, then $S$ will consist of weighted composition operators,  which are much more complicated to handle.   
    We refer to \cite[Thm. 2.6]{matache} and \cite[Lem.2.1]{cp03} for more details concerning similarity transforms.  
     \\

\textbf{Acknowledgments:}  The authors thank Eva Gallardo-Guti\'errez and Dimitri Yakubovitch for several comments on a  preliminary version of this paper.   The authors are also grateful to the referee for valuable comments concerning general domains instead of the open unit disc. 

\bibliographystyle{amsplain}

\begin{thebibliography}{99}
\bibitem{AERS} D. Aharonov, M. Elin, S. Reich and D. Shoikhet, \emph{Parametric representations of semi-complete vector fields on the unit balls of $\mathbb{C}^n$ in Hilbert space}, Atti Acad. Naz. Linci, 10 (1999), 229-233.\\ 	
	
\bibitem{delhi1} W. Arendt, I. Chalendar, M. Kumar and S. Srivastava, \emph{Asymptotic behaviour of the powers of composition operators on Banach spaces of holomorphic functions}, Indian Math. J., to appear.  \\	
	
\bibitem{ACP2015} C. Avicou, I. Chalendar and J.R. Partington, \emph{A class of quasicontractive semigroups acting on Hardy and Dirichlet space}, J. Evolution Equations, 15 (2015), 647-665.\\
	
\bibitem{ACP2016} C. Avicou, I. Chalendar and J.R. Partington, \emph{Analyticity and compactness of semigroups of composition operators}, J. Math. Anal. and Appl., 437 (2016), 545-560.\\
%
%
%
%
%
	
	\bibitem{BP} E. Berkson and H. Porta, \emph{Semigroups of analytic functions and composition operators}, Michigan Math. J. 25 (1978), no. 1, 101--115.\\
	
	\bibitem{team} O. Blasco, M.D.  Contreras, S. Diaz-Madrigal, J. Martinez, M. Papadimitrakis and A.G. Siskakis, \emph{Semigroups of composition operators and integral operators in spaces of analytic functions}, Annales Academiae Scientiarum Fennicae Mathematica 38 (2013), 1--23.\\
	
	\bibitem{cp03}
	I. Chalendar and J.R. Partington, \emph{On the structure of invariant subspaces for isometric composition operators on $H^2(\D)$ and $H^2(\C_+)$}, Arch. Math. 81 (2003), 193--207.\\
	
	\bibitem{CM} C.C. Cowen, B.D. MacCluer, \emph{Composition Operators on Spaces of Analytic Functions}, CRC Press, Boca
	Raton, FL, 1995.\\
	
 \bibitem{GY} E. Gallardo-Guti\'errez and D. Yakubovitch, \emph{On generators of $C_0$-semigroups of composition operators}, arXiv:1708.02259. \\
	
	\bibitem{ms74} M. Hirsch and S. Smale, \emph{Differential Equations, Dynamical Systems and Linear Algebra}, Academic Press, New-York, 1974.\\
	

	\bibitem{MW} W. Hurewicz, \emph{Lectures on Ordinary Differential Equations}, Dover publications, New-York, 2014.\\
	

	\bibitem{matache} V. Matache, \emph{Composition Operators on Hardy spaces of a half-plane}, Proc. Amer. Math. Soc., 127, No. 5, 1999, 1483-1491.\\	
%
%
%
%
%
%
%
%
\bibitem{RS2005} S. Reich and D. Shoikhet, \emph{Nonlinear Semigroups, Fixed points, and Geometry of Domains in Banach Spaces}, Imperial College Press, London, 2005.\\ 

\bibitem{shoik} D. Shoikhet,  \emph{Semigroups in geometrical function theory}, Kluwer Academic Publishers, Dordrecht, Boston, London, 2001.\\
%
%

%
%
%

  
\bibitem{siskakis} A.G. Siskakis, \emph{Semigroup of composition operators on spaces of analytic functions, a review}, Studies on composition operators, 229--252, Contemp. Math., 213, Amer.
Math. Soc., Providence, RI, 1998.\\	
%
%
%
%
%
%
\end{thebibliography}

\end{document}